\documentclass{daj}

\usepackage[latin9]{inputenc}
\usepackage{amsmath}
\usepackage{amsthm}
\usepackage{amssymb}
\usepackage{setspace}

\numberwithin{equation}{section}
\def\R{\mathbb R}
\def\Z{\mathbb Z}
\def\C{\mathbb C}
\def\P{\mathbb P}
\def\N{\mathbb N}
\def\E{\mathbb E}

\def\ee{\varepsilon}

\def\gcd{\operatorname{gcd}}
\def\deg{\text{deg}}

\newtheorem*{theorem*}{Theorem}
\newtheorem{theorem}{Theorem}[section]
\newtheorem{lemma}[theorem]{Lemma}

\theoremstyle{remark}
\newtheorem{remark}[theorem]{Remark}

\theoremstyle{definition}
\newtheorem{definition}[theorem]{Definition}

\theoremstyle{remark}

\numberwithin{equation}{section}

\dajAUTHORdetails{%
  title = {Irreducibility of Random Polynomials of Bounded Degree}, 
  author = {Huy Tuan Pham, and Max Wenqiang Xu},
  plaintextauthor = {Huy Tuan Pham, and Max Wenqiang Xu},
}   

\dajEDITORdetails{%
   year={2021},
   number={7},
   received={26 February 2020},   
   published={20 July 2021},  
   doi={10.19086/da.25316},       
}   

\begin{document}

\begin{frontmatter}[classification=text]

\title{Irreducibility of Random Polynomials of Bounded Degree} 

\author[huypham]{Huy Tuan Pham\thanks{Supported by the Craig Franklin Fellowship in Mathematics}}
\author[maxxu]{Max Wenqiang Xu\thanks{Supported by the Cuthbert C. Hurd Graduate Fellowship in Mathematics}}

\begin{abstract}
It is known that random monic integral polynomials of bounded degree $d$ and integral coefficients distributed uniformly and independently in $[-H,H]$ are irreducible over $\Z$ with probability tending to $1$ as $H\to \infty$. In this paper, we give a general criterion for guaranteeing the same conclusion under much more general coefficient distributions, allowing them to be nonuniformly and dependently distributed over arbitrary sets. 
\end{abstract}
\end{frontmatter}

\section{Introduction}

A lot of research has been  done in studying the irreducibility and other arithmetic properties of random integral polynomials. In general, there are two different models often considered, namely, the bounded degree model, e.g., see \cite{van}, \cite{chela}, \cite{Sam}, \cite{kozmanagain}, \cite{igor}, \cite{man}, \cite{chern}; and the bounded height model, e.g., see \cite{BV}, \cite{bary}, \cite{Phil}, \cite{kon}, \cite{PV2}. An excellent survey of this active research area can be found in \cite{wood}. In the bounded height model, the polynomial $f(x)=x^d+a_{d-1}x^{d-1}+\dots+a_0$ has growing degree $d\to \infty$ while the coefficients $a_i$'s are chosen independently and uniformly at random from a set of fixed size. In the bounded degree model, the polynomial has fixed degree $d$ while the integral coefficients $a_i$'s are chosen independently and uniformly at random from $[-H,H]$ with $H\to \infty$. Most of the previous results in the bounded degree model applied only to this specific distribution of the coefficients. In this paper we focus on studying irreducibility over $\Z$ of random polynomials  in a wide class of bounded degree models. 

\subsection{Background} Van der Waerden  \cite{van} proved that the probability of  irreducibility over $\Z$ of a random polynomial where all integral coefficients are distributed independently and uniformly at random in $[-H,H]$ tends to 1. Chela \cite{chela} found the tight bound $(1+o(1))C_d/H$ (where $C_d$ is a constant only depending on degree $d$) for the probability that such a polynomial is reducible.

It would be interesting to know if this phenomenon of irreducibility over $\Z$ of random polynomials is more general, as the previous proofs depend heavily on the fact that the integral coefficients are drawn uniformly and independently at random from the interval $[-H,H]$. However, it does not seem to be the case that irreducibility over $\Z$ depends heavily on the exact distributions of the coefficients. In fact, it is intuitive that for a generic distribution of random integral polynomials, the random polynomial is irreducible with high probability. Some numerical evidence that supports this can be found in \cite{wood}. Indeed, in this paper, we show that under much more general distributions of the coefficients, the polynomial is reducible with small probability. 

The methods used in \cite{van} and \cite{chela} heavily used the multiplicative property of the Mahler measure (see \cite{BG} for details), which leads to the multiplicative property of the height of polynomials, where the height of an integral polynomial is the largest absolute value of the coefficients. When the integral coefficients are uniformly and independently distributed in an interval,  this multiplicative property of the height allows us to efficiently constrain the possible factors of the polynomial (if there exists one), from which we can count the number of polynomials that admit a nontrivial factorization over $\Z$. However, when the integral coefficients are not uniformly distributed in an interval, for example, when they are distributed on a sparse subset of a big interval, then the height is much larger than the size of the support of the coefficients, and the above argument does not give good bounds. In this paper, we will introduce new methods to cover such cases, even allowing for the support of the coefficients to be arbitrary sets.

In \cite{igor}, Rivin introduced a method to show that random polynomials are irreducible over $\Z$ with high probability by exploiting the cases where the constant coefficient has few divisors. This method can be used to quickly obtain van der Waerden's result. If the constant integral coefficient is distributed uniformly in $[-H,H]$, then it is not hard to show it has very few divisors on average. This can be used to prove an upper bound $O(\log H/H)$ on the probability that the polynomial is reducible. A version of this method was utilized by Bary-Soboker and Kozma  \cite{kozmanagain} and Chern \cite{chern} to prove irreducibility of random integral polynomials generalizing the uniform distribution over $[-H,H]$ to other distributions of random polynomials which satisfy certain conditions. In particular, both papers consider the case where the coefficients are chosen independently and uniformly from a set $S$ of integers such that there is some small prime $p$ (not much larger than $|S|$) for which the elements of $S$ are distinct modulo $p$. They then combine the factorization of the constant coefficient over $\mathbb{Z}$ together with the factorization of the polynomial over $\mathbb{Z}/p\Z$ to deduce the bounds on the probability of reducibility. 

It is crucial in \cite{kozmanagain} and \cite{chern} that the coefficients are distributed independently and further satisfied strong congruence relations modulo a small prime. Thus, these results do not apply to the case of coefficients distributed over sets that are not well-behaved modulo $p$, such as the set of polynomial values of an integral polynomial, or to the case of dependent coefficients. 

In another direction, one can ask if similar results hold when only a subset of the coefficients is randomized. Cohen \cite{Cohen2variable, Cohen2variablewrong} (together with \cite{CohenGalois}) showed that the probability of irreducibility tends to one when all but two coefficients are fixed and appropriate conditions hold. In fact, Cohen's results apply to the Galois group of the polynomial. However, they come with weaker probability bounds that are suboptimal under our settings, and require the coefficients to be uniformly and independently distributed in a box.

Our main contribution in this paper is to address the above restrictions of the previous methods. In particular, we give nearly optimal bounds for the probability of reducibility even when we allow for the coefficients to have nontrivial dependency and non-uniform distribution over sparse sets. Our results also hold even if we only randomize one appropriate coefficient of the polynomial.
By working directly in a field of charateristic $0$ instead of mapping to a finite field, we completely avoid the congruence restrictions on the distribution of the coefficients.
Our technique is also based on the surprising observations of \cite{igor}. However, we crucially make explicit and quantitative some objects in the argument. Via this explicit construction, we verify the key assumption for the validity of the argument by using complex analysis. 

\subsection{Main Result}

We consider random polynomials with random coefficients where the degree $d$ of the polynomial is fixed and the distribution of the coefficients has support growing to infinity. We show that, under mild assumptions on the joint distribution of the coefficients, the probability that the random polynomial is irreducible over $\Z$ tends to 1 as the support of the coefficient distribution tends to infinity. In particular, we show that this holds even when the coefficients are not independently or uniformly distributed. Moreover, we allow the coefficients to be supported on arbitrary subsets of integers with size tending to infinity instead of the interval $[-H,H]$. In fact, our result applies even when only one appropriate coefficient is random, while the remaining coefficients are fixed. These give significant generalizations of previous results in the literature. We remark that we do not need the degree $d$ of the polynomial to be fixed, but our bounds are meaningful only when $d$ is quite small compared to the size of the support of the coefficient distributions. 

Our method considers separately the distribution of the constant coefficient and the distributions of the higher-degree coefficients. For the constant coefficient, we only need to require its expected number of divisors to be small. For higher-degree coefficients, we allow $a_i$'s to be a polynomial change of variable from a product distribution that is not too far from the uniform distribution, satisfying a certain non-degeneracy condition. This relaxes both the uniformity and the independence of the coefficients. In all cases, we allow the support of the coefficients $a_i$ to be completely arbitrary, significantly generalizing the bounded box model. For appropriate instances, we also show that it suffices to have one random coefficient $a_i$ for $i$ coprime to the degree $d$ (or a subset of coefficients $a_i,i\in I$ for $\gcd(\{d\}\cup I)=1$) while the rest of the coefficients are fixed or arbitrarily distributed. 

Our most general result, Theorem~\ref{strong theorem}, gives a bound for the probability that the random polynomial is reducible under a very general class of coefficient distributions. Theorem~\ref{strong theorem} involves checking a mild non-degeneracy condition that is not straightforward to verify. Theorem~\ref{markov} gives a special case of Theorem~\ref{strong theorem} where we can verify the condition, while still applying to general coefficient distributions. In the following two theorems, Theorem~\ref{thm:simple cases} and Theorem~\ref{thm:vary-one-coeff}, we describe some straightforward specific applications of Theorem~\ref{markov} and Theorem~\ref{strong theorem} in simplified settings, where already previous results in the literature do not apply. We expect that the general result, Theorem~\ref{strong theorem}, could be applicable to many more interesting situations. 

\begin{theorem}\label{thm:simple cases}
Let $f(x)=x^d+a_{d-1}x^{d-1}+\dots+a_1x+a_0\in \mathbb{Z}[x]$, where $a_0,\dots,a_{d-1}$ are distributed randomly according to one of the following distributions. 
\begin{enumerate}
    \item Sparse coefficient distribution: $a_0$ is uniformly distributed over the set of polynomial values $P(n)$ for $n\in [-H,H]$ and $P$ not the zero polynomial, the remaining $a_j$ for $j \ne 0$ are uniformly distributed over an arbitrary set of size $H$, and $a_0,a_1,\dots,a_{d-1}$ are independent. Then there exists a constant $C(P)$ only depending on $P(x)$ such that $$\mathbb{P}(f(x)\textrm{ is irreducible over $\mathbb{Z}$}) \ge 1 - \frac{\deg(P)}{2H+1} - \frac{2d2^d (\log {H})^{C(P)}}{H}.$$
    \item Nonuniform coefficient distribution: $a_j$ is distributed according to the binomial distribution $B(H,p)$ for $\min(Hp,H(1-p)) \to \infty$, and $a_0, a_1, \dots, a_{d-1}$ are independent. Then 
    $$\mathbb{P}(f(x)\textrm{ is irreducible over $\mathbb{Z}$}) \ge 1 - o(1).$$
    \item Dependent coefficient distribution: $a_0$ is uniformly distributed over the set of polynomial values $P(n)$ for $n\in [-H,H]$ and $P$ not the zero polynomial, and for $i\in \{1,2,3,\cdots,d-1\}$, $a_{i}=P_i(a_{i-1},\dots,a_{0},t_{i})$ where $P_i$ are polynomials of degree at most $L$ and $P_i$ has positive degree in $t_i$ for any fixed nonzero value of $a_0$, and $t_1,\dots,t_{d-1}$ are independently and uniformly chosen from an arbitrary set of size $H$. Then there exists a constant $C(P)$ only depending on $P(x)$ such that $$\mathbb{P}(f(x)\textrm{ is irreducible over $\mathbb{Z}$}) \ge 1 - \frac{\deg(P)}{2H+1} - \frac{2 d^2(2L)^d (\log {H})^{C(P)}}{H}.$$
\end{enumerate}
\end{theorem}

In Case (1) of Theorem~\ref{thm:simple cases}, if we choose $P(x)=x$, and $a_j$ distributed uniformly over $[-H,H]$, then we recover the classical model. Already in this case, we observe that the probability bound we provide is optimal up to polylogarithmic factors in $H$. By relaxing congruence conditions in \cite{kozmanagain, chern}, we can allow for arbitrary distribution of $a_j$, for example, over polynomial values of any given integral polynomial. In \cite{wood}, a question is raised about irreducibility of polynomials whose coefficients are distributed according to the Binomial distribution. This motivates us to consider non-uniform coefficient distributions, of which Case (2) of Theorem~\ref{thm:simple cases} is an example. Case (3) of Theorem~\ref{thm:simple cases} is an example where the coefficients are not independently distributed. In general, the most general version of our result, Theorem~\ref{strong theorem}, allows for the coefficients $(a_0,\dots,a_{d-1})$ to be any polynomial function of a set of independent random variables $(t_0,\dots,t_m)$. 

In many situations we can guarantee the conclusion of Theorem~\ref{thm:simple cases} under much weaker assumptions, where only one appropriate coefficient or an appropriate subset of coefficients of the polynomial is varied while the remaining coefficients are fixed or distributed according to an arbitrary distribution (similar to the setting of \cite{Cohen2variable, Cohen2variablewrong}). We denote by $\tau(n)$ the number of divisors of $n$. 

\begin{theorem}\label{thm:vary-one-coeff}
Let $f(x)=x^d+a_{d-1}x^{d-1}+\dots+a_1x+a_0\in \mathbb{Z}[x]$, where $a_0,\dots,a_{d-1}$ are distributed randomly according to the following distribution. For some $h\ge 1$ coprime to $d$, $a_h$ is uniformly distributed in an arbitrary set of size at least $H$, the remaining $a_j$ for $j\ne h$ have an arbitrary distribution, and $a_0,a_1,\dots,a_{d-1}$ are independent. 

Then $$\mathbb{P}(f(x)\textrm{ is irreducible over $\mathbb{Z}$}) \ge 1 - \mathbb{P}(a_0=0) - \frac{2d2^d \mathbb{E}[\tau(a_0)]}{H}.$$

Furthermore, the same conclusion holds if for a subset $I$ of $[d-1]$ with $\gcd(\{d\}\cup I)=1$, $a_i$ is uniformly distributed in an arbitrary set of size at least $H$ for $i\in I$, the remaining $a_j$ for $j\notin I$ have an arbitrary distribution, and $a_0,a_1,\dots,a_{d-1}$ are independent. 
\end{theorem}

Theorem~\ref{thm:vary-one-coeff} shows that $f$ is irreducible with high probability even when we fix arbitrarily all coefficients apart from one appropriate coefficient, or an appropriate subset of coefficients.

We next introduce several useful definitions before stating a more general result, Theorem~\ref{markov}, from which Theorems \ref{thm:simple cases} and \ref{thm:vary-one-coeff} can be deduced.

\begin{definition} We call a positive integer $n$ is {\it $s$-simple} if $\tau(n) \le s$.
\end{definition}

In particular, any integer with bounded number of prime factors (counted with repeat, that is, $\Omega(n)$), is $s$-simple for $s$ being a constant. Even for $s = 2$, infinitely many integers (primes) are $s$-simple. All sufficiently large integers $n$ are $n^{\ee}$-simple and $n$ is $s$-simple if $\Omega(n)\le \log s$.

We also quantify the notion of being close to the uniform distribution. 

\begin{definition}
Let $D$ be a distribution supported on a finite subset of $\mathbb{C}$ of size $k$. We say that $D$ is {\it $(C,t)$-uniform} if for any subset $T$ of $\mathbb{C}$ of size at most $t$, $D(T) \le \frac{C|T|}{k}$. 
\end{definition}

The uniform distribution over $k$ points is $(1,k)$-uniform. In general, it is easy to see that any distribution on $k$ points whose largest probability of a point is $p$ is $(pk,k)$-uniform.

The next theorem gives a general class of coefficient distribution where we can give a quantitative bound on the probability of irreducibility. The statement of our most general result, Theorem~\ref{strong theorem}, is delayed to Section~\ref{main}. 

\begin{theorem}\label{markov}
Let $(t_1,\dots,t_m)$ be independent random variables. 
Let 
$f(x) =  x^{d} +a_{d-1}x^{d-1} + \dots + a_1 x + a_0$ in $\Z[x]$.
Assume that for some $0<k<d$ coprime to $d$, we have 
\begin{align*}
    a_k &= F_2(a_0,t_1,\dots,t_m),\\
    (a_j)_{1\le j\le d-1,j\ne k} &= F_1(a_0,t_1,\dots,t_{m-1}),
\end{align*}
where $F_2:\mathbb{R}^{m+1} \to \mathbb{R}$ and $F_1:\mathbb{R}^{m}\to \mathbb{R}^{d-2}$ are polynomials of degree at most $L$ and $F_2$ has positive degree in $t_m$ for any fixed nonzero value of $a_0$. Assume that $t_i$ is distributed according to a $(C,2^dL)$-uniform distribution over a set $H_{i}$ with size at least $H$. Then 
$$\P(f(x) ~\text{is irreducible over $\Z$})  \ge 1-\mathbb{P}_{D_0}(a_0=0) - \frac{2Cdm2^d L\mathbb{E}[\tau(a_0)]}{H}.$$

In particular:
\begin{enumerate}
    \item If $a_0$ is distributed uniformly on values of an nonzero integral polynomial $\{P(n): -H\le n \le H\}$, then there exists a constant $C(P)$ only depending on $P(x)$ such that \[\P(f(x) ~\text{is irreducible over $\Z$}) \ge 1 - \frac{\text{deg}(P)}{2H+1}- \frac{2Cmd 2^dL(\log {H})^{C(P)}}{H} .\]
    \item If $a_0$ is supported on $s$-simple integers, then \[\P(f(x) ~\text{is irreducible over $\Z$}) \ge 1 - \frac{2Cmd 2^d Ls}{H}. \]
\end{enumerate}
\end{theorem}

As we mentioned before, for the constant coefficient, our method only requires it to have a small number of divisors on average. This is true for uniform distribution over polynomial values of any given integral polynomial or any distribution over $s$-simple integers.

Theorem~\ref{markov} is a corollary of our most general result, Theorem~\ref{strong theorem}, which allows for more general coefficient distribution. As evident from the proof later given, we actually show that the probability that there exists a factorization $f=gh$, where $g,h$ are polynomials with complex coefficients whose constant coefficients are integral, is small. Since a polynomial is reducible over $\mathbb{Z}$ only if such a decomposition exists, this gives an upper bound on the probability that the polynomial is reducible over $\mathbb{Z}$. Theorem~\ref{strong theorem} involves a non-degeneracy condition, which essentially captures cases when the probability of having a factorization into complex polynomials with integral constant coefficients is small. This condition can be verified in the setting of Theorem~\ref{markov}, although it should be true in more general settings.  

Our proof also suggests that significant new insights would be needed in order to deal with models where the non-degeneracy condition fails to hold, since considering the factorization of the constant coefficient alone would not suffice to provide a nontrivial bound on the probability of being reducible. See Remark \ref{mistake} for more details on an interesting situation where this condition does not hold. 

\subsection{Strategy}  The proof of Theorem~\ref{strong theorem} consists of two parts.
First, we transfer the question about studying irreducibility of the random polynomial $f$ to studying the number of zeros of the polynomials $P_{d,k,a_0,b_0}$ defined in Definition~\ref{P}. This idea was first used by Rivin in \cite{igor}. Here we provide an explicit construction of $P$ in Lemma~\ref{factorize-polynomial}, instead of using an elimination process in \cite{igor}, which is crucial to our verification of the non-degeneracy condition that a certain transformation of $P$ remains nonzero. More discussions about this non-degeneracy condition can be found in Remark~\ref{mistake}. We complete the proof by using a version of Schwartz-Zippel Lemma,  Lemma~\ref{SZ}, to bound the number of zeros over $\mathbb{C}$ of multivariate polynomials. 

Theorem~\ref{markov} is a straightforward application of Theorem~\ref{strong theorem}. In the proof of Theorem \ref{markov}, we need to check the non-degeneracy condition of a transformation of $P_{d,k,a_0,b_0}$, which is studied in Theorem~\ref{check-not-zero} via tools from complex analysis. 

\subsection{Organization of the Paper} We organize the paper as following.  
We state and prove the most general result, Theorem~\ref{strong theorem}, in Section \ref{main}. Then, in Section \ref{models}, we prove Theorem~\ref{markov},  Theorem~\ref{thm:simple cases}, and Theorem~\ref{thm:vary-one-coeff}, which give explicit applications of the main theorem. The key step in the proof of Theorem~~\ref{markov} is the verification of the non-degeneracy condition in Theorem~\ref{strong theorem}, which is the content of Theorem~\ref{check-not-zero} in Section~\ref{models}.

\subsection{Notation}
We adopt the usual asymptotic notations. We write $X =o_{x \to \infty}(Y)$ if for $|X|/|Y|\to 0$ as $x\to \infty$. When the dependency on $x$ is clear, we may omit the subscript. We write $X=O(Y)$ if there exists a constant $C$ such that $|X|\le C |Y|$.

Throughout this paper, $\N$ is the set of positive integers. We define $[n]: = \{1,2,\dots,n\}$. The divisor function $\tau(n)$  counts the number of divisors of an integer $n$. For a random variable $a=(a_1,a_2,\dots,a_d)$, we write $a \sim D$ to denote that $a$ has distribution $D$. We also denote by $\mathbb{P}_{D}, \mathbb{E}_{D}$ the probability or expectation with respect to the distribution $D$. When it is clear form the context, we often omit $D$ from the notations. We use $I$ to denote the characteristic function.

\section{A general criterion for irreducibility of random polynomials} \label{main}

In this section, we prove Theorem~\ref{strong theorem}, which is a general sufficient criterion for irreducibility over $\Z$ of random polynomials. Before stating the theorem, we need some preparations.

Let $f(x)  =  x^{d} + a_{d-1} x^{d-1}  + \dots + a_1 x +a_0 $ in $\C[X]$.
We will introduce some explicit polynomials which can decode some information about the factorization of $f(x)$. 

\begin{definition}\label{P}
For any positive integers $d, k$, and integers $a_0,b_0$, define $$P_{d,k,a_0,b_0}(a_1,\dots,a_{d-1}):=\prod_{1\le i_1<i_2<\dots<i_k\le d}(b_0-\alpha_{i_1}\alpha_{i_2}\dots\alpha_{i_k}),$$ where $\alpha_1,\alpha_2,\dots,\alpha_d$ are the (complex) roots of $f(x)=x^d+a_{d-1}x^{d-1}+\dots+a_1x+a_0$ in $ \C[X]$.  
\end{definition}

We remark that the polynomial $P_{d,k,a_0,b_0}$ defined above depends only on $d,k,a_0,b_0$, since the product in the definition of $P_{d,k,a_0,b_0}$ is a symmetric polynomial in the roots. Furthermore, the degree of $P_{d,k,a_0,b_0}$ is at most $\binom{d}{k} < 2^d$.  

The following lemma relates the polynomial $P$ defined above to the reducibility over $\Z$ of the polynomial $f$.
 
\begin{lemma} \label{factorize-polynomial}
Fix complex numbers $a_0,b_0 \neq 0$. The polynomial $f(x)$ in $\C[x]$ with degree $d$ can be factorized as a product of polynomials $g(x)$ and $h(x)$ in $\C[x]$ such that $\deg(g)=k$ and the constant coefficient of $g$ is $b_0$ if and only if $P_{d,k,a_0,b_0}(a_1,\dots,a_{d-1})=0.$ Here $k$ is a positive integer smaller than $d$. 
\end{lemma}

\begin{proof}
	Let $\alpha_1, \alpha_2, \dots, \alpha_d$ be the roots of polynomial $f(x)$. Then
$f(x) = g(x) h(x)$ with fixed $a_0, b_0$ is equivalent to there is a product of $k$ roots of $f$ which is equal to $b_0$, i.e., 
\[ \prod_{1\le i_1 < i_2 < \dots < i_k \le d} (b_0 - \alpha_{i_1} \alpha_{i_2} \dots \alpha_{i_k}) = 0. \]
	
	By definition, $P_{d,k,a_0,b_0}(a_1,\dots,a_{d-1})=0.$

\end{proof}

Now we can state our main theorem which gives sufficient conditions for a random polynomial to be irreducible over $\Z$ with high probability.  

\begin{theorem} \label{strong theorem}
Let $L$ be a positive integer. Let $f(x) = x^{d}  +a_{d-1}x^{d-1}+\dots+a_0$ in $\Z[x]$ with positive degree $d>1$. Let $D_t$ be a product distribution on $T_0 \times \dots\times T_{m} \subset \Z^{d}$ such that the distribution on each $T_i$ is $(C,2^dL)$-uniform, where $C$ is a positive real number, and $|T_i| = k_i > 0$. Suppose that the following two conditions hold:

\begin{enumerate}
    \item There is a deterministic polynomial function $F:\mathbb{R}^{m+1} \to \mathbb{R}^{d-1}$ of degree at most $L$ such that the distribution $D_a$ is the push-forward of $D_t$ under the map $(t_0,t_1,\dots,t_{m}) \mapsto (a_0, a_1, \dots, a_{d-1}): =(t_0,F(t_0,\dots,t_{m}))$.
    \item For all fixed positive integers $k < d,$ and $a_0 \ne 0$ and $b_0|a_0$, all the polynomials
    $$P_{d,k,a_0,b_0}(F(a_0,t_1,\dots,t_{m}))$$ are not identically zero, where $P_{d,k,a_0,b_0}$ is defined explicitly as in Definition~\ref{P}. 
\end{enumerate}

Then the probability that $f(x)=x^d+a_{d-1}x^{d-1}+\dots+a_0$ is irreducible over $\Z$ is at least $$1-\mathbb{P}_{D_0}(a_0=0) - 2Cd2^d L\tau(D_0)\sum_{i=1}^{m} \frac{1}{k_i},$$ where $D_0$ is the distribution of $a_0$, and $\tau(D_0): = \mathbb{E}_{D_0}[\tau(a_0)I(a_0\ne 0)]$.
\end{theorem}

Comparing to Theorem~\ref{markov},
Theorem~\ref{strong theorem} applies more generally to  distributions of coefficients which are a polynomial change of variables from a product and close-to-uniform distribution that satisfies a further mild non-degeneracy condition.

To prove Theorem~\ref{strong theorem}, we use Lemma \ref{SZ}, which is a version of the Schwartz-Zippel Lemma which bounds the number of zeros of a multivariate polynomial on a grid over $\C$. The proof is a straightforward generalization of the original proof of the Schwartz-Zippel Lemma, see \cite{schwartz} and \cite{zippel}.

For a polynomial $P(x_1,x_2,\dots,x_m)$, we let $d_m$ be the degree of $x_m$ in $P$. We can then write $P$ as $$P(x_1,x_2,\dots,x_m) = x_m^{d_m}P_{m}(x_1,\dots,x_{m-1}) + Q_m(x_1,\dots,x_m),$$ where the degree of $x_m$ in $Q_m$ is at most $d_m-1$. Having defined $d_m$ and $P_m$, for $2\le i\le m$, we then recursively define $d_{i-1}$ to be the degree of $x_{i-1}$ in $P_i$, and define $P_{i-1}$ so that $$P_i(x_1,\dots,x_{i-1})=x_{i-1}^{d_{i-1}} P_{i-1}(x_1,\dots,x_{i-2}) + Q_{i-1}(x_1,\dots,x_{i-1}),$$ where the degree of $x_{i-1}$ in $Q_{i-1}$ is at most $d_{i-1}-1$. 

\begin{lemma} \label{SZ}
Let $P(x_1,x_2,\dots,x_m)$ be a polynomial of degree $d$ which is not identically $0$ and define $d_i$ as above for $i\in [m]$. Let $T_1, T_2,\dots,T_m$ be $m$ subsets of $\mathbb{C}$ such that $|T_i|=k_i$. Let $D_i$ be a distribution on $T_i$ which is $(C,d_i)$-uniform. Let $D=\prod_{i=1}^{n}D_i$. Then $$\mathbb{P}_{(x_1,\dots,x_m)\sim D}(P(x_1,\dots,x_m)=0) \le \sum_{i=1}^{m}\frac{Cd_i}{k_i}.$$ 
\end{lemma}

\begin{proof}
We prove this by induction on $m$. The result is clearly true when $m=1$, as a nonzero univariate polynomial of degree $d$ has at most $d$ roots. 

Assuming the result is true for all $m\le h$, and consider $m=h+1$. Write 
\[P(x_1,\dots,x_{h+1}) = P_{h+1}(x_1,\dots,x_{h}) x_{h+1}^{d_{h+1}} + \sum_{i<d_{h+1}} Q_{h+1,i}(x_1,\dots,x_{h}) x_{h+1}^i,\] 
where $P_{h+1}$ and $Q_{h+1,i}$ are polynomials in $x_1,\dots,x_{h}$, and $P_{h+1}$ is not identically $0$. If $(x_1,\dots,x_h)$ is so that $P_{h+1}(x_1,\dots,x_h) \ne 0$, then there are at most $d_{h+1}$ values of $x_{h+1}$ so that $P(x_1,\dots,x_{h+1}) = 0$. Thus, conditioned on $P_{h+1}(x_1,\dots,x_h) \ne 0$, the probability that $P(x_1,\dots,x_{h+1}) = 0$ is at most $\frac{Cd_{h+1}}{k_{h+1}}$. Hence, $$\mathbb{P}_{(x_1,\dots,x_{h+1})\sim D}(P(x_1,\dots,x_{h+1})=0) \le \mathbb{P}_{(x_1,\dots,x_{h+1})\sim D}(P_{h+1}(x_1,\dots,x_h)=0) + \frac{Cd_{h+1}}{k_{h+1}}.$$ By the inductive hypothesis, $$\mathbb{P}_{(x_1,\dots,x_{h+1})\sim D}(P_{h+1}(x_1,\dots,x_h)=0) \le \sum_{i=1}^h\frac{Cd_i}{k_i},$$ which gives $$\mathbb{P}_{(x_1,\dots,x_{h+1})\sim D}(P(x_1,\dots,x_{h+1})=0) \le \sum_{i=1}^{h+1}\frac{Cd_i}{k_i}.$$
Thus the conclusion is true for $m = h+1$ as well, and we complete the proof.
\end{proof}

Observe that $\sum_{i=1}^m d_i \le \text{deg}(P)$. Hence, in the special case where $k_i=k$ for all $i\in [m]$, we obtain $$\mathbb{P}_{(x_1,\dots,x_m)\sim D}(P(x_1,\dots,x_m)=0) \le \frac{C \deg(P)} {k}.$$ Also, if the degree of $P$ is small, we can bound $d_i\le \deg(P)$ and obtain the more convenient bound $$\mathbb{P}_{(x_1,\dots,x_m)\sim D}(P(x_1,\dots,x_m)=0) \le \sum_{i=1}^m \frac{C\deg(P)}{k_i}.$$

Next we give the proof of Theorem~\ref{strong theorem}. 
\begin{proof}[Proof of Theorem~\ref{strong theorem}]
If $a_0=0$ then $f(x)$ is reducible over $\Z$.

Consider $a_0\ne 0$. Suppose $f(x)$ is reducible, then there exists a positive integer $1\le k \le d-1$ such
\[f(x) = g(x) h(x), \]
where 
\[g(x) = \sum_{j=0}^{k} b_jx^j ~\text{and}~ h(x) = \sum_{j=0}^{d-k} c_jx^j ~~~\text{in }~\Z[X].\]
 This gives a factorization $a_0 = b_0c_0$, and there are at most $\tau(a_0)$ choices of $b_0$ and $c_0$. We fix a choice of $b_0$. By Lemma \ref{factorize-polynomial}, since $f(x)$ admits a factorization into $g(x)h(x)$ with $g$ having degree $k$ and constant coefficient $b_0$, we must have $P_{d,k,a_0,b_0}(a_1,\dots,a_{d-1})=0$ for some polynomial $P_{d,k,a_0,b_0}$ depending on $d,k,a_0,b_0$ of degree at most $2^d$. Note that we even allow the coefficients of $g$ except the constant coefficient to be any complex number. This gives a polynomial $Q(t_1,\dots,t_{m})=0$ via assumption (1), where $Q$ has degree at most $\deg (P) \deg(F) \le 2^d \deg (F)$ and $Q$ is not identically zero by assumption (2). By Lemma \ref{SZ}, the probability that $f$ admits such a factorization is at most $\sum_{i=1}^{m} \frac{C2^d\deg(F)}{k_i}$. Finally, by taking the union bound over $b_0$ and $k$, we get for nonzero $a_0$,
 $$\P(f ~\text{is reducible}\,|\,a_0) \le 2\tau(a_0) \sum_{i=1}^{m} \frac{Cd2^d \deg(F)}{k_i}\le 2\tau(a_0) \sum_{i=1}^{m} \frac{Cd2^d L}{k_i}.$$ 
Hence, we conclude that
\begin{align*}
\P (f ~\text{is irreducible}) &\ge 1-\mathbb{P}_{D_0}(a_0=0)- 
2\mathbb{E}_{D_0}\left[\tau(a_0)I(a_0\ne 0) \sum_{i=1}^{m} \frac{Cd2^dL}{k_i}\right] \\
&= 1- \mathbb{P}_{D_0}(a_0=0) - 2Cd2^dL \tau(D_a) \sum_{i=1}^{m} \frac{1}{k_i}.  
\end{align*}
 
\end{proof}

\section{Models of random polynomials} \label{models}

\subsection{Models of the constant coefficient}
In this subsection, we discuss various models of the distribution $D_0$ of the constant coefficient $a_0$ such that $\tau(D_0) = \mathbb{E}_{D_0}[\tau(a_0)I(a_0\ne 0)]$ is not large. 
\subsection*{Model 1: Uniform distribution on integral polynomial values}

We use the following result due to van der Corput \cite{divisor}.

\begin{lemma}[van der Corput]\label{corput}
Suppose that $P(x)$ is a polynomial with integral coefficients, then 
\[\sum_{-H\le n \le H: P(n) \neq 0} \tau(P(n)) \le H (\log {H})^{C(P)}.   \]
Here $C(P)$ is a constant which only depends on the polynomial $P(x)$.  
\end{lemma}
In other words, if $D_0$ is the uniform distribution on the polynomial values $\{ P(n): -H \le n \le H\}$,  then there exists a constant $C(P)$ depending on $P$ such that
\[\tau(D_0) =  \E_{D_0}[\tau(a_0)I(a_0\ne 0)] \le  (\log {H})^{C(P)}.    \]

Notice that when $P(x) = x$, this is the same as the classical model where the coefficient $a_0$ is chosen uniformly at random from $[-H,H]$.

\subsection*{Model 2: Arbitrary distributions supported on integers with few divisors}

By our definition of $s$-simple integers, we immediately obtain the following result. 

\begin{lemma}\label{powers lemma}
For any distribution $D_0$ supported on $s$-simple integers, 
\[\tau(D_0) =  \E_{D_0}[\tau(a_0)I(a_0\ne 0)]\le s.    \]
\end{lemma}

\subsection{Models of random polynomials}

In this section, we apply our Theorem~\ref{strong theorem} to various models of random polynomials and give several concrete applications of our main theorem. In particular, we will prove Theorem~\ref{markov}.   

In order to apply Theorem \ref{strong theorem} to specific models, the main condition to verify is that $$P_{d,k,a_0,b_0}(F(a_0,t_1,\dots,t_{d-1}))$$ is not identically zero. 
We show that this is the case when $F$ is of the following form: for some integer $h \le d$ which is coprime to $d$, \begin{align*}
    (a_j)_{1\le j\le d,j\ne h}=(a_1,\dots,a_{h-1},a_{h+1},\dots,a_d) &= F_1(t_0,t_1,\dots,t_{m-1}),\\
    a_h &= F_2(t_0,t_1,\dots,t_{m-1},t_m).
\end{align*} Here, $F_1$ is a polynomial mapping $\R^{m-1}$ to $\R^{d-2}$, and $F_2$ is a polynomial mapping $\R^{m}$ to $\R$ such that $$F_2(t_0,t_1,\dots,t_{d-1}) = \sum_{i=0}^{\deg(t_{m})} P_i(t_0,t_1,\dots,t_{m-1})t_{m}^i,$$ and not all $P_i(t_0,t_1,\dots,t_{m-1})$ for $i\ge 1$ are polynomials only in $t_0$.

\begin{theorem}[Checking non-degeneracy]\label{check-not-zero}
Let $h \le d$ be coprime to $d$. Assume that $a_0, a_1, \dots, a_{d-1}$, $t_0, t_1, \dots, t_{m-1}$ are variables taking values in $\R$ such that $a_0 = t_0$. Let $F_1$ and $F
_2$ be two fixed polynomials from $\R^{m-1}$ to $\R^{d-2}$ and $\R^{m}$ to $\R$ respectively, such that $(a_j)_{1\le j\le d-1,j\ne h} = F_1(t_0,t_1,\dots,t_{m-1})$ and $a_{h} = F_2(t_0,t_1,\dots,t_{m})$, where $$F_2(t_0,t_1,\dots,t_{m}) = \sum_{i=0}^{\deg(t_{m})} P_i(t_0,t_1,\dots,t_{m-1})t_{m}^i,$$ and not all $P_i(t_0,t_1,\dots,t_{m-1})$ for $i\ge 1$ are polynomials only in $t_0$. 

Then the polynomial $P_{d,k,a_0,b_0}((F_1(a_0,t_1,\dots,t_{m-1}), F_2(a_0,t_1,\dots,t_{m})))$ in the variables $t_1,\dots,t_{m}$ defined as in Definition~\ref{P} is not identically zero for all integers $1\le k <d$,  and nonzero $a_0$ and $b_0$. 
\end{theorem}

\begin{proof}
By the definition of $P_{d,k,a_0,b_0}$, it suffices to show that there exists a fixed choice of $t_1=s_1, t_2 =s_2, \dots,t_{m-1}=s_{m-1}$, such that for some $t_{m}$ the polynomial $f(x) = x^d + a_{d-1}x^{d-1} + \dots a_1x+ a_0$ does not have $k$ roots whose product is $b_0$. 

Since $F_2(t_0,t_1,\dots,t_{m}) = \sum_{i=0}^{\deg(t_{m})} P_i(t_0,t_1,\dots,t_{m-1})t_{m}^i$ and not all $P_i(t_0,t_1,\dots,t_{m-1})$ for $i\ge 1$ is a polynomial in $t_0$, there exists a choice of $t_1=s_1,\dots,t_{m-1}=s_{m-1}$ such that $F_2(a_0,s_1,\dots,s_{m-1},t_{m})$ is not identically zero as a polynomial in $t_{m}$. Fix such choice of $t_1,\dots,t_{m-1}$. Since we fixed $a_0,t_1,\dots,t_{m-1}$, the coefficients $a_j$ for $j\ne h$ are fixed constants. Since $F_2(a_0,s_1,\dots,s_{m-1},t_{m})$ is not identically zero, it is a univariate polynomial in $t_{m}$, hence $|a_{h}|$ tends to infinity as $t_{m}$ tends to infinity. 

We prove that for fixed $a_j,j\ne h$ and $k,b_0$, for all $|a_{h}|$ sufficiently large, the polynomial $f(x) = x^d + a_{d-1}x^{d-1} + \dots+ a_1x+ a_0$ does not have $k$ roots whose product is $b_0$. 

In particular, we first show that for any $\ee \in (0,1)$ sufficiently small, if $|a_{h}|$ is sufficiently large in $\epsilon$, for each root $\alpha$ of the equation $x^{d-h}+\alpha=0$, $f$ has a root within $\ee|a_h|^{1/(d-h)}$ of $\alpha$, and furthermore, the remaining $h$ roots of $f$ are within $\ee$ of $0$.

By the argument principle, the number of roots of the polynomial $f$ in the region bounded by a circle $C$ in the complex plane is given by 
\begin{align*}
\frac{1}{2\pi i}\int_C \frac{f'(x)}{f(x)}dx &= \frac{1}{2\pi i} \int_C \frac{dx^{d-1}+\sum_{j=0}^{d-2}(j+1)a_{j+1}x^{j}}{x^d+\sum_{j=0}^{d-1}a_{j}x^{j}}dx \\
&= \frac{1}{2\pi i}\int_C \frac{hx^{h-1}+\frac{1}{a_{h}}(dx^{d-1}+\sum_{j\in[0,d-2],j\ne h}(j+1)a_{j+1}x^{j})}{x^{h}+\frac{1}{a_{h}}(x^d+\sum_{j\in[0,d-1],j\ne h}a_{j}x^{j})}dx.
\end{align*}

Let $\delta>0$ be arbitrary. Letting $C = \{x:|x|=\ee\}$, if $|a_{h}|$ is large enough (in $\ee$ and $\delta$), we have
for all $x \in C$, 
$$\left | \frac{hx^{h-1}+\frac{1}{a_{h}}(dx^{d-1}+\sum_{j\in[0,d-2],j\ne h}(j+1)a_{j+1}x^{j})}{x^{h}+\frac{1}{a_{h}}(x^d+\sum_{j\in[0,d-1],j\ne h}a_{j}x^{j})} - \frac{h}{x} \right| < \delta.$$  

Thus, $$\left |\frac{1}{2\pi i}\int_C \frac{f'(x)}{f(x)}dx - \frac{1}{2\pi i}\int_C \frac{h}{x}dx\right |<\delta.$$ Notice that $$\frac{1}{2\pi i}\int_C \frac{h}{x}dx=h,$$ and $\frac{1}{2\pi i}\int_C \frac{f'(x)}{f(x)}dx$ is an integer, we have $$\frac{1}{2\pi i}\int_C \frac{f'(x)}{f(x)}dx = h.$$ 
Hence $f$ has $h$ roots whose magnitude is at most $\ee$. 

Next, let $C=\{x:|x-\alpha|=\ee |\alpha|\}$ where $\alpha$ is a root of $x^{d-h}+a_h=0$. Observe that $|\alpha| = |a_h|^{1/(d-h)}$. Assume that $\ee$ is sufficiently small. For $\delta>0$ arbitrary, if $|a_h|$ is large enough in $\ee$ and $\delta$, we claim that for $x\in C$,
\[
\left| \frac{hx^{h-1}+\frac{1}{a_{h}}dx^{d-1}+\frac{1}{a_{h}}\sum_{j\in[0,d-2],j\ne h}(j+1)a_{j+1}x^{j}}{x^{h}+\frac{1}{a_{h}}x^d+\frac{1}{a_{h}}\sum_{j\in[0,d-1],j\ne h}a_{j}x^{j}} - \frac{hx^{h-1}+dx^{d-1}/a_h}{x^h+x^d/a_h} \right| < \delta.
\]
To see this, note that for $x \in C$, we can write $x= \alpha + \ee y$ with $|y|=|\alpha|=|a_h|^{1/(d-h)}$. For $\ee$ sufficiently small and $a_h$ sufficiently large, 
\[
|x^h+x^d/a_h| = |x|^h |(a_h+(\alpha+\ee y)^{d-h})/a_h| \ge \frac{1}{2}|\alpha|^h \cdot \ee |y| |\alpha|^{d-h-1}/|a_h| = \frac{1}{2}\ee |a_h|^{h/(d-h)}.
\]
Since $\left|\frac{1}{a_h}x^{j}\right| \le 2|a_h|^{(h-1)/(d-h)}$ for $j\notin \{d,h\}$ and $x\in C$, we can easily obtain the desired claim. Since 
\[
\frac{1}{2\pi i}\int_{C} \frac{hx^{h-1}+dx^{d-1}/a_h}{x^h+x^d/a_h}dx = 1,
\]as the function $g(x)=x^{h}+x^d/a_h$ has a unique root $\alpha$ in the region bounded by $C$, we have that 
\[
\frac{1}{2\pi i}\int_{C} \frac{hx^{h-1}+\frac{1}{a_{h}}(dx^{d-1}+\sum_{j\in[0,d-2],j\ne h}(j+1)a_{j+1}x^{j})}{x^{h}+\frac{1}{a_{h}}(x^d+\sum_{j\in[0,d-1],j\ne h}a_{j}x^{j})}dx = 1.
\]
Hence, there is one root of $f$ within $\ee|a_h|^{1/(d-h)}$ of each root $\alpha$ of $x^{d-h}+a_h=0$. 

The roots $\beta$ of $f$ with $|\beta|<\ee$ satisfy $a_h\beta^h = -a_0(1+o_{a_h\to \infty}(1))$, so $|\beta| = (1+o_{a_h\to \infty}(1))|a_0/a_h|^{1/h}$. 
Thus, the polynomial $f$ has $d-h$ roots of magnitude $(1+o_{a_h\to \infty}(1))|a_0/a_h|^{1/h}$, and $h$ roots of magnitude $(1+o_{a_h\to \infty}(1))|a_h|^{1/(d-h)}$. For any $k$ roots of $f$, if there are $k_1$ roots of magnitude $(1+o_{a_h\to \infty}(1))|a_0/a_h|^{1/h}$ and $k_2$ roots of magnitude $(1+o_{a_h\to \infty}(1))|a_h|^{1/(d-h)}$, then the product of the roots has magnitude $(1+o_{a_h\to \infty}(1)) |a_0|^{k_1/h} |a_h|^{k_2/(d-h)-k_1/h}$. For $|a_h|$ sufficiently large, in order for this product to be equal to $b_0 \ne 0$, we need to have $k_2/(d-h)=k_1/h$. However, since $\gcd(d,h)=1$, this can only happen when $k=k_1+k_2=d$, contradiction. 

Thus, there exists a choice of $t_1,\dots,t_{d-1}$ so that $f$ does not have $k$ roots whose product is $b_0$.
\end{proof}

\begin{remark}\label{mistake}
We remark that while the nondegeneracy condition of $P_{d,k,a_0,b_0}(F(a_0,t_1,\dots,t_{m}))$ being not identically $0$ is likely generically true, it is important to verify this condition carefully. Indeed, there are some interesting families of random polynomials where it turns out that this condition fails to hold. In such cases, the technique of this paper breaks down and cannot give nontrivial bounds on the probability that the random polynomial is reducible over $\mathbb{Z}$. Hence, in such cases, considering only the factorization of the constant coefficient over $\mathbb{Z}$ is not sufficient to show that the polynomial is irreducible over $\mathbb{Z}$ with high probability. We next highlight one such case.

Let $f(x) = x^{d} +a_{d-1}x^{d-1} + \dots a_1 x +a_0$ where each integer $a_i \in [-H, H]$. Let $\alpha_1, \alpha_2,\dots, \alpha_d$ be the (complex) roots of $f$.
A famous conjecture of van der Waerden from 1936 states that the Galois group $G_f$ of polynomial $f$ is equal to symmetric group $S_d$ with probability $1-O(1/H)$.

Consider the following polynomial 
\[f_6(x) = \prod_{1\le i_1<i_2<\dots < i_6\le d} (x -\alpha_{i_1}\alpha_{i_2} \dots \alpha_{i_6}).\]
This is polynomial of degree $\binom{d}{6}$ in $x$ with constant coefficient $a_0^{\binom{d-1}{5}}$. We denote each of its roots $\beta_i$.  
For $d\ge 12$, it turns out that Galois Group of $f$ is $6$-transitive (i.e., equals to $S_d$ or $A_d$) if and only if $f_6(x)$ is irreducible over $\Z$.

Thus, to resolve van der Waerden's conjecture, it would be interesting to first show that $f_6$ is irreducible with probability $1-O(1/H)$. This was discovered by Rivin in \cite{igor}, and motivated us to consider the irreducibility of random polynomials with more general coefficient distributions. Unfortunately, if one wants to directly use the method in this paper, then it would fail since $$P_{\binom{d}{6}, 6d, a_0^{\binom{d-1}{5}}, a_0^6}(F(a_0,t_1,\dots,t_{d-1}))$$ is identically $0$. This is because there always exists a subset of tuples of $6$ roots of $f$ of size $d$ whose product is $a_0^6$. Thus the probability that there exists a factorization of $f=gh$ as complex polynomials where the constant coefficients of $g$ and $h$ are integers is $1$. This is a particular example where considering only the factorization of the constant coefficient is not sufficient to get nontrivial bounds on the probability of reducibility. 
\end{remark}

We next give the proofs of Theorem~\ref{thm:simple cases}, Theorem~\ref{thm:vary-one-coeff} and Theorem~\ref{markov}, which readily follow from Theorem~\ref{strong theorem} and Theorem~\ref{check-not-zero}.

\begin{proof}[Proof of Theorem~\ref{markov}]
Theorem~\ref{markov} follows directly from Theorem \ref{strong theorem}, Theorem \ref{check-not-zero}, Lemma \ref{corput} and Lemma \ref{powers lemma}.
\end{proof} 

\begin{proof}[Proof of Theorem~\ref{thm:simple cases}]
\noindent Case (1): This follows directly from Theorem~\ref{markov} with $L=1$ and $C=1$. 

\noindent Case (2): This follows from Theorem~\ref{markov} and the observation that the binomial distribution $B(H,p)$ 
is $((1+o_{pH \to \infty}(1))\sqrt{H/(2\pi p(1-p))},H)$-uniform. 

\noindent Case (3): This follows from Theorem~\ref{markov}, noting that we have 
\begin{align*}
a_{d-1} &= F_2(a_0,t_1,\dots,t_{d-1}),\\
(a_j)_{1\le j\le d-2} &= F_1(a_0,t_1,\dots,t_{d-2}),
\end{align*}
where the polynomials $F_1,F_2$ have degree at most $L^d$. 
\end{proof} 

\begin{proof}[Proof of Theorem~\ref{thm:vary-one-coeff}]
The first claim of Theorem~\ref{thm:vary-one-coeff} follows from Theorem~\ref{markov} and Lemma \ref{corput}.

To show the second claim, we need to verify the non-degeneracy condition in Theorem~\ref{strong theorem} via a modification of Theorem~\ref{check-not-zero}. The proof is similar to the proof of Theorem~\ref{check-not-zero}, we only give a sketch of the key modifications in the following. 

Let $I$ be a subset of $[d-1]$ for which $\gcd(\{d\}\cup I)=1$. Arrange the elements of $I$ in decreasing order $i_1>i_2>\dots>i_s$. By fixing all coefficients outside $I$ (so that $a_0\ne 0$), and considering an asymptotic where $a_{i_1}\to \infty$ and $\omega_{t} := a_{i_{t}}/a_{i_t-1} \to \infty$ at appropriate rates: $\omega_{t}$ tends to infinity sufficiently slow in $\omega_{t-1}$ for $3\le t\le s$, and $\omega_{2}$ tends to infinite sufficiently slow in $|a_{i_1}|$. By a similar argument to the proof of Theorem \ref{check-not-zero}, we can show that for each $2\le t\le k$, there are $i_{t-1}-i_{t}$ roots of $f(x)=x^d+a_{d-1}x^{d-1}+\dots+a_0$ that are of order $\omega_{t}^{1/(i_{t-1}-i_t)}$. Furthermore, there are $d-i_1$ roots of order $|a_{i_1}|^{1/(d-i_1)}$, and $i_s$ roots of order $|a_{i_s}|^{-1/i_s}$. Then, as in Theorem \ref{check-not-zero}, we can verify that for any $k<d$, it cannot be the case that there always exists a subset of $k$ roots whose product is equal to a fixed constant. This implies the second claim in Theorem~\ref{thm:vary-one-coeff}. 
\end{proof} 

\section*{Acknowledgments} 

The second author got introduced to this topic by his Part III essay advisor Professor P\'eter Varj\'u when he was a student at Trinity College, University of Cambridge. The supervision is gratefully acknowledged. The authors would like to thank him for many helpful conversations and suggestions on this topic. We would like to thank Professor Melanie Matchett Wood for helpful discussions and careful readings on earlier drafts.
We would like to thank Professor Kannan Soundararajan for valuable comments on earlier drafts. We thank the anonymous referee and editor for detailed and useful feedback which improved the presentation and results of this paper.

\bibliographystyle{amsplain}
\bibliography{randompoly}{}

\begin{dajauthors}
\begin{authorinfo}[huypham]
  Huy Tuan Pham\\
  Department of Mathematics, \\
  Stanford University, Stanford, CA 94305, USA\\
  huypham\imageat{}stanford\imagedot{}edu \\
\end{authorinfo}
\begin{authorinfo}[maxxu]
  Max Wenqiang Xu\\
  Department of Mathematics, \\
  Stanford University, Stanford, CA 94305, USA\\
  maxxu\imageat{}stanford\imagedot{}edu 
\end{authorinfo}
\end{dajauthors}

\end{document}